\documentclass[11pt]{article}

\usepackage{t1enc}
\usepackage{amsthm,amssymb}
\usepackage{amsmath, eucal}
\usepackage{latexsym}
\usepackage{amsmath} 
\usepackage{graphics}
\usepackage{epsfig,multicol}
\usepackage{amsfonts}
\usepackage[latin1]{inputenc}
\usepackage[english]{babel}
\usepackage{relsize} 

\usepackage{relsize} 
\begin{document}

\newtheorem{theorem}{Theorem}[section]
\newtheorem{corollary}[theorem]{Corollary}
\newtheorem{lemma}[theorem]{Lemma}
\newtheorem{proposition}[theorem]{Proposition}
\newtheorem{conjecture}[theorem]{Conjecture}

\newtheorem*{coro}{Corollary} 

\theoremstyle{definition}
\newtheorem{definition}[theorem]{Definition}
\newtheorem{example}[theorem]{Example}
\newtheorem{remark}[theorem]{Remark}

\def\hpic #1 #2 {\mbox{$\begin{array}[c]{l} \epsfig{file=#1,height=#2}
\end{array}$}}
 
\def\vpic #1 #2 {\mbox{$\begin{array}[c]{l} \epsfig{file=#1,width=#2}
\end{array}$}}

\newcommand{\OF}{$\vec F $ }
\newcommand{\HO}{HOMFLYPT }
\newcommand{\FO}{$ \vec {\mathcal F} $ }
\newcommand{\rotovercrossing}{\mathbin{\rotatebox[origin=c]{-90}{$\overcrossing$}}}
\newcommand{\rotundercrossing}{\mathbin{\rotatebox[origin=c]{-90}{$\undercrossing$}}}

\title{The Homflypt polynomial and the oriented Thompson group.}
\author{Valeriano Aiello$^\dag$, Roberto Conti$^{\flat}$, Vaughan F. R. Jones$^\sharp$\thanks{V.J. is supported by the NSF under Grant No. DMS-0301173 and grant DP140100732, Symmetries
of subfactors.}\\
\\
$^\dag$ Dipartimento di Matematica e Fisica\\ Universit\`a Roma Tre \\ 
Largo S. Leonardo Murialdo 1, 00146 Roma, Italy.\\
e-mail: valerianoaiello@gmail.com\\
\\
$^\flat$ Dipartimento di Scienze di Base e Applicate per l'Ingegneria \\ 
Sapienza Universit\`a di Roma \\ Via A. Scarpa 16, 00161 Roma, Italy.\\
e-mail: roberto.conti@sbai.uniroma1.it\\
\\
$^\sharp$ Department of Mathematics, \\
Vanderbilt University, \\
1362 Stevenson Center, Nashville, TN 37240, USA\\
e-mail: vaughan.f.jones@vanderbilt.edu
}
\maketitle
\begin{abstract}
 We show how to construct  unitary representations of the oriented Thompson group $\vec F $ from 
 oriented link invariants. In particular we show that the suitably normalised HOMFLYPT polynomial
 defines a positive definite function of \OF. 
 \end{abstract}

\section{Introduction}
The Thompson group $F$ is usually defined geometrically as a group of piecewise linear homeomorphisms of the circle. One can then
readily give a combinatorial description in terms of pairs of planar rooted binary trees, see \cite{CFP}. In \cite{jo1} the geometrical
definition of $F$ was used to obtain some knot and  link TQFT invariants 
as coefficients of a specific ``vacuum'' vector $\Omega$, for certain unitary representations of $F$. The \emph{calculation} of these coefficients used a direct construction of a link
from the combinatorial description of an element of $F$ as a pair of binary trees, 
by replacing each vertex of the trees by a crossing, where the
roots at the top and bottom are vertices, and the loose ends introduced are all tied
togther in the only planar way possible. 

If $g\in F$ let us call $\mathcal L (g)$ 
the corresponding link, $L\mapsto Q(L)$  the relevant TQFT link invariant  ($L$ denotes a generic link) and  $\pi$ the unitary representation of $F$ defined in \cite{jo1}.
The picture below illustrates how to construct $\mathcal L(g)$ from a pair of trees:

$$
\includegraphics[scale=0.4]{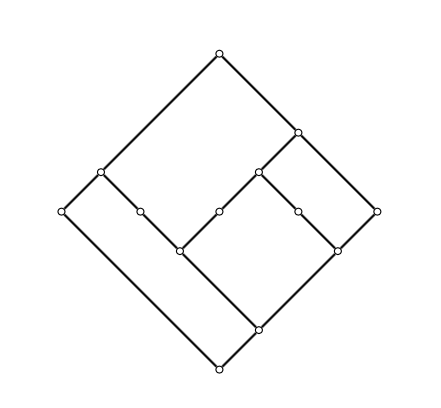}\qquad \quad
\includegraphics[scale=0.4]{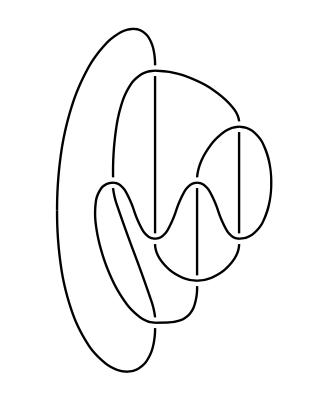}
$$
where each vertex is replaced by the following tangle as shown below:
$$
\includegraphics[scale=0.4]{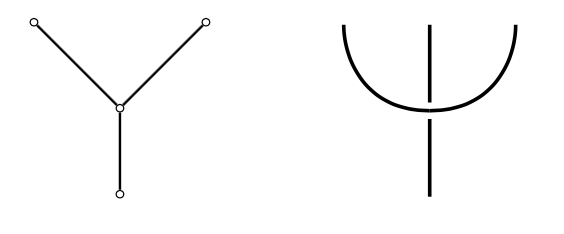}
$$

and the free ends are connected in the obvious way.

Then $$\langle \pi(g)\Omega,\Omega\rangle = Q(\mathcal L(g)).$$
 
Since $\pi$ is unitary, the map $g\mapsto Q(\mathcal L(g))$ is a positive definite function on $F$, which is not
obvious without this interpretation.

In \cite{jo1} it was shown that \emph{any} link arises as $\mathcal L(g)$ for some $g\in F$ so that $F$ is as good as the braid groups
at producing links. Unfortunately $\mathcal L(g)$ receives no natural orientation from $g$. This defect was overcome in \cite{jo1}
by restricting the construction  to a subgroup $\vec F<F$ for which the chequerboard shading surface for  $\mathcal L(g)$ is oriented
and so defines an orientation of $\mathcal L(g)$ as the oriented boundary. It was further shown that all \emph{oriented} links
arise in this way from $\vec F$.

But \cite{jo1} failed to interpret \emph{oriented} TQFT invariants as coefficients of unitary representations of $\vec F$.

In this paper we fill this gap, following the observation by the first two authors (extending the result in  \cite{AC2}) that, at least for HOMFLYPT, the map
$g\mapsto Q(\mathcal L(g))$ is in fact positive definite on $\vec F$. For a precise statement, see Theorem \ref{HOMFLYPT} below.

We will achieve this interpretation using the method of \cite{jo2} which is a remarkably flexible way of constructing actions of
many groups by realising them as \emph{groups of fractions} $G(\mathfrak C)$ (in the sense of Ore) for categories $\mathfrak C$ with stabilisation and
cancellation properties: every functor from $\mathfrak C$ to another category $D$ admitting direct limits yields immediately an action of
$G(\mathfrak C)$ on a certain direct limit. 

In the case of Thompson's group $F$ we have $F=G(\mathfrak C)$ where $\mathfrak C$ is the category $\mathfrak F$ of binary planar rooted forests.
There is an abundance of readily obtainable functors from $\mathfrak F$ to other categories, in particular one may realise the construction
of $\mathcal L(g)$ by a functor to the category of Conway tangles. In this paper we will show how to realise $\vec F$ 
as the field of fractions of an oriented version $\vec{ \mathfrak F}$ of $\mathfrak F$ (see Proposition \ref{OF}). And \emph{oriented} TQFT invariants $Q$ will give rise to functors
from  $\vec {\mathfrak F}$ which give unitary representations $\pi$ of $\vec F$ such that $$Q(\mathcal L(g))=\langle \pi(g)\Omega,\Omega\rangle$$
for $g\in \vec F$. 
Positive definiteness of the map $g\mapsto Q(\mathcal L(g))$ becomes immediate.

\section{Positivity of the HOMFLYPT inner product.}

Let $\sigma$ be a function from $\{1,2,\cdots , 2n\}$ to $\{+,-\}$. such that $|\sigma^{-1}(+)| =|\sigma^{-1}(-)|=n$. Let $s\in \mathbb C$ and $k \in \mathbb N$ be given. We consider the vector
spaces $V_\sigma$ given by the HOMFLYPT skein module for the circle with $2n$ boundary points. That is to say the quotient of the space of all linear 
combinations of
Conway tangles of oriented links  (orientations compatible at the boundary with $\sigma$ with $+=$ out and $- =$ in) modulo the skein
relation \\
$$
\includegraphics[scale=0.25]{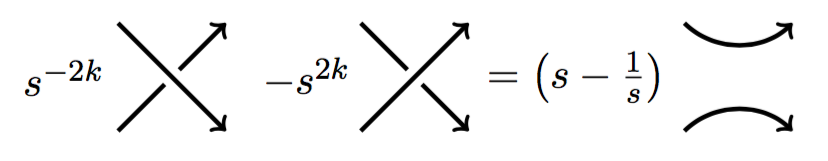}
$$

These vector spaces have a natural *-structure given by reflection, reversing all the orientations and the condition 
$$
\includegraphics[scale=0.25]{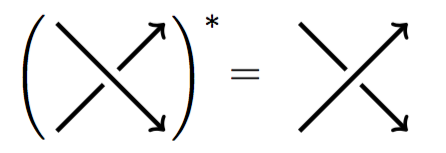}
$$

One needs to specify a first point on the boundary of a diagram
as well. See \cite{jo3} for more details. The HOMFLYPT polynomial equips the $V_\sigma$ with scalar products $\langle \;, \;\rangle$ where if $T_1$ and $T_2$
are tangles, $\langle T_1,T_2\rangle$ is the \HO polynomial of the link: 
$$
\includegraphics[scale=0.35]{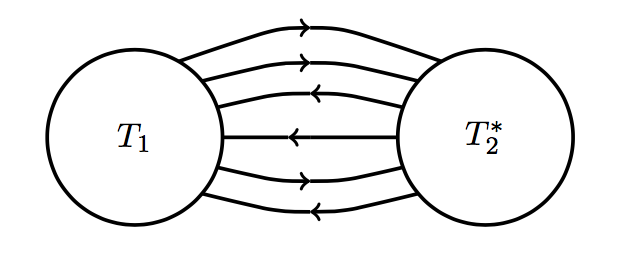}
$$

\begin{proposition} \label{prop-inner-prod}
The inner product $\langle \; ,\; \rangle$ on $V_\sigma$ defined above is positive semidefinite if $s=e^{\pi i/r}$ for $r\geq k+2, r\in \mathbb Z$.
\end{proposition}
\begin{proof} This is a corollary of Wenzl's paper \cite{wn} 
in which the result is proved for the special choice of $\sigma$ with $\sigma=(+,+,+.\cdots,+,-,-,-,\cdots,-)$. 
In this case, it is well known that the the corresponding tangles are linearly spanned by braids (see e.g. \cite{MoTr}). 
The connection between the 
HOMFLYPT polynomial and Wenzl's result in \cite{wn}, which is formulated in terms of Hecke algebras,
is explained quite in detail in \cite[Section 3]{wn2}. We only observe some change in conventions and notations: consider the change of parameter $(r,s) \mapsto (r,-s^{-1})$ in \cite{wn2}
and then Theorem 3.1.1 therein applies with $r = s^{-2k}$ (the Hecke algebra parameter becomes $q=s^{-2}$).

To get from this choice of boundary orientations to any other one, first observe that the permutation group $S_{2n}$ is transitive on all the $\sigma$. So it suffices
to show 
that positivity is conserved if the boundary orientation $\sigma$ is changed to $\sigma'$ by applying a transposition between adjacent elements.  For this simply
surround the circle in which the tangles live with a single crossing 
as: \\
$$
\includegraphics[scale=0.25]{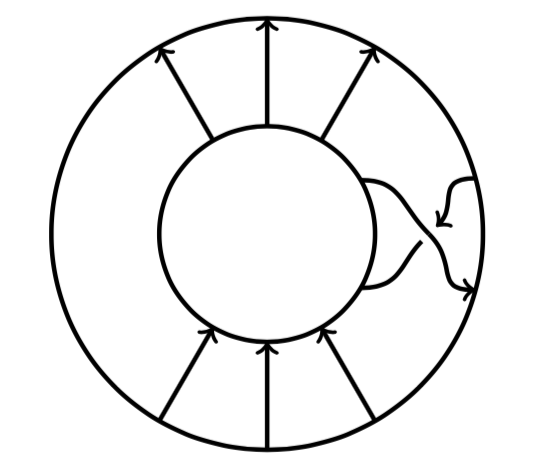}
$$
This  defines a $\langle \; ,\; \rangle$-preserving linear map from $V_\sigma$ to $V_{\sigma'}$.
The invariance of the inner product is a consequence of the application of 
a Reidemeister move of type II, as shown in the figure below for a concrete example. 
Consider the tangles $T_1, T_2 \in V_{(+++---)}$ given by
$$
\includegraphics[scale=0.35]{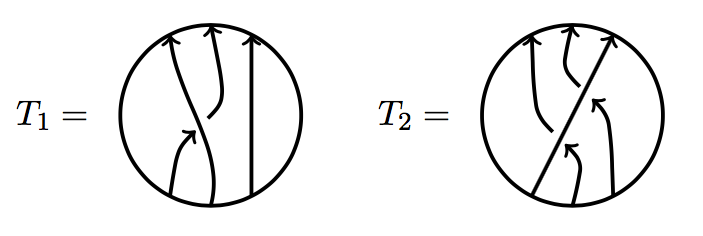}
$$
then the scalar product $\langle T_1,T_2\rangle$ is the HOMFLYPT polynomial of the link
$$
\includegraphics[scale=0.5]{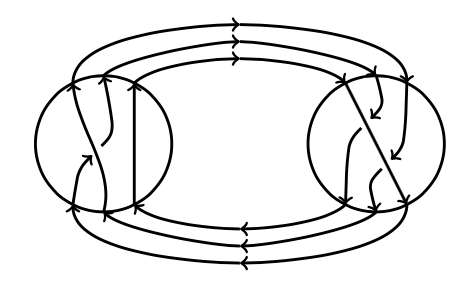}
$$
while the scalar product of the images of $T_1, T_2$ 
in $V_{(++-+--)}$ is the HOMFLYPT polynomial of the link
$$
\includegraphics[scale=0.35]{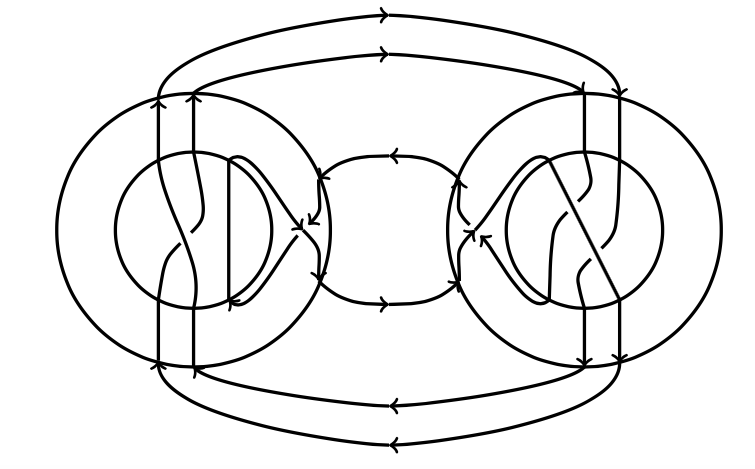}
$$
The invariance of the inner product in the other cases follows from similar arguments. 
\end{proof}
\begin{remark} \label{real} If one changes the *-structure by making an oriented crossing self-adjoint, the inner product becomes positive
semidefinite when $s\in \mathbb R^+$.
\end{remark}
\section{ The category $\vec{\mathfrak F}$ and its group of fractions. }
Recall from \cite[Section 5.3.2]{jo1} that a pair of binary planar trees representing an element of $F$ determines a surface 
with boundary as below: 
$$
\includegraphics[scale=0.25]{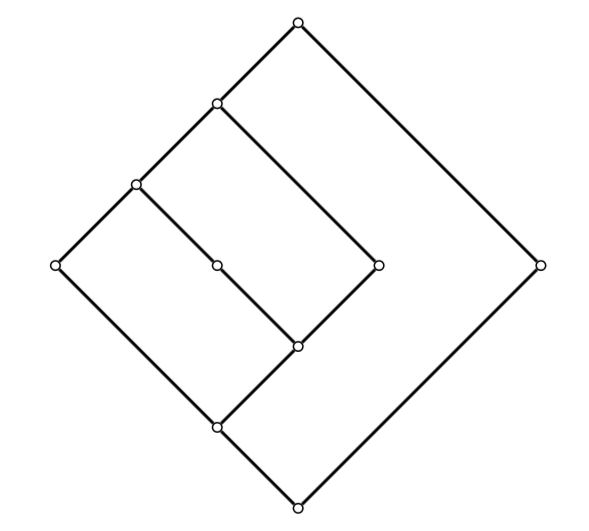}
$$

$$
\includegraphics[scale=0.25]{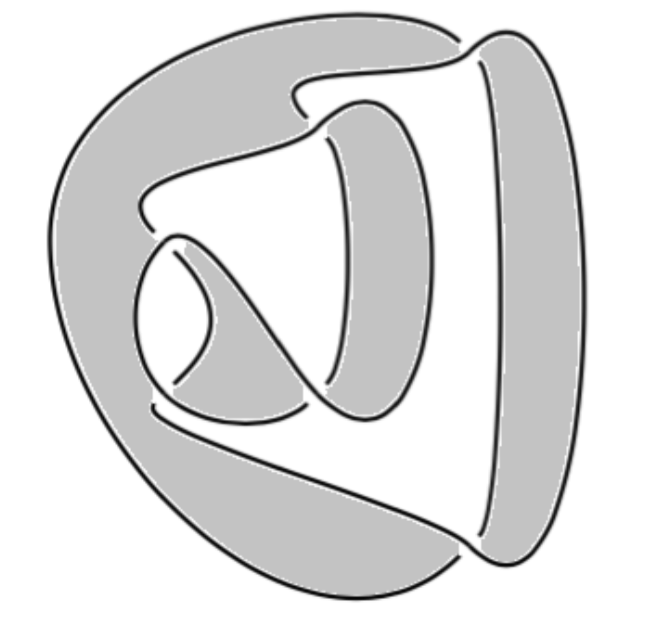}
$$

   $\vec F$ is the subgroup of $F$ for which this  surface is orientable and one chooses an orientation $+$ for the
leftmost part of the surface in the picture. The link which is the boundary of the surface is then oriented with the boundary orientation. It is shown
in \cite{jo1} that, modulo distant unknots, every oriented link arises in this way from an element of \OF.

In \cite[Section 2.1]{jo2} a group of fractions was associated to an object $|$ in a category ${\it Cat}$ having three properties, namely
\begin{enumerate}
\item (Unit) ${\it Mor}_{\it Cat}(|,a)\neq \emptyset$ for all $a\in {\it Ob}({\it Cat})$.
\item (Stabilisation) Let $\mathcal{D}:=\cup_{a\in {\it Ob}({\it Cat})}{\it Mor}_{\it Cat}(|,a)$. Then for each $f$, $g\in \mathcal{D}$ there are morphisms $p$ and $q$ such that $pf=qg$.
\item (Cancellation) Given two morphisms $p$ and $q$, if $pf=qf$ for $f\in \mathcal{D}$ then $p=q$. 
\end{enumerate} 
 If the category has one object the group of fractions is due to Ore. More importantly an
action of the group of fractions is associated to any functor $\Phi$ from ${\it Cat}$ to another category ${\it Kat}$. In the case where ${\it Cat}$ is the category of planar binary 
forests, $|$ is the forest with one root and one leaf, the group of fractions is Thompson's group $F$ and morphisms are obtained whenever ${\it Kat}$  has a morphism
which behaves like a tree with one root and two leaves. In particular in the category of  unoriented tangles a crossing provides such a morphism. Applying a unitary
unoriented TQFT to the picture gives unitary representations of $F$ on a Hilbert space with a special ``vacuum'' vector $\Omega$ so that the coefficient $\langle g\Omega, \Omega\rangle$ 
is the invariant  that the TQFT assigns to the link. We will now repeat this construction by constructing a category $\vec{\mathfrak F}$ of planar binary forests with extra orientation data. 
$\vec{\mathfrak F}$ will satisfy the axioms of \cite{jo2} and its group of fractions will be obviously isomorphic to \OF. Moreover assigning \emph{oriented} crossings to the forks in $\vec{\mathfrak F}$ will give a
functor to the category of tangles and applying a unitary oriented TQFT such as SU(N) at some level will give a unitary representation of \OF whose coefficients $\langle g\Omega, \Omega\rangle$  produce the
HOMFLYPT polynomial of the oriented link constructed from $g$ in \cite{jo1}.

\begin{definition} For $n=1,2,3,\cdots$ an \emph{$n$-sign} will be a sequence of $n$ $+$'s and $-$'s 
such that the first sign is $+$ and the second is $-$ (if $n\geq 2)$. \end{definition}
An $n$-sign determines a bi-colouring (by $\pm$) of the regions to the left of  $n$ points on $\mathbb R$ so that  the infinite left region is coloured $+$ (adjacent 
regions other than the leftmost two may have the same colour).

Binary planar forests are as in \cite{jo1}. They form a category under stacking.
\begin{proposition}
Given an $m$-sign $\sigma$ and a binary planar forest $f$ whose $m$ roots lie on $\mathbb R$ and whose $n$ leaves lie on
a parallel copy of $\mathbb R$, the colouring of the regions between the roots of $f$ extends uniquely to a colouring of the regions of the
planar map given by the forest, with the property that regions sharing an edge of $f$   emanating clockwise from the lower incident edge at a triple point
have
different colours. 
Thus $f$ and $\sigma$ together determine an $n$-sign $f(\sigma)$ which can be read off from the copy of $\mathbb R$ containing the leaves of $f$.
\end{proposition}
 We give a proof by example.
 Consider $\sigma$ and $f(\sigma)$ equal to $(+,-,+,-)$ and  $(+,-,-,+,-,+,-,-,+,+)$ respectively, then
 $$
\includegraphics[scale=0.25]{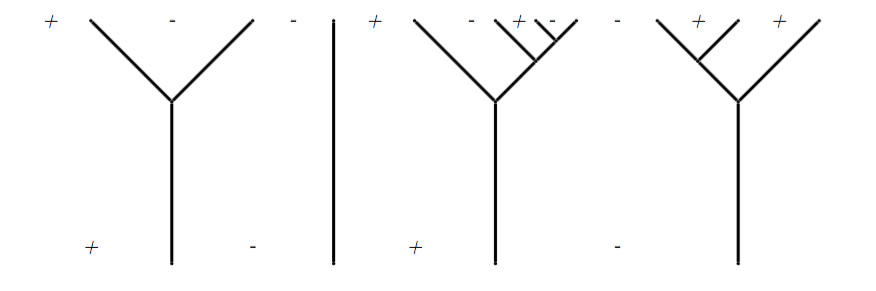}
$$

 \begin{proposition}
 The map $\sigma\rightarrow f(\sigma)$ is functorial, i.e. $f(g(\sigma))=(f\circ g)(\sigma)$.
 \end{proposition}
 \begin{proof} 
The claim is actually a consequence of the uniqueness of the procedure described in the former proposition. In fact, on the one hand we have a colouring of the regions determined by $g$ and induced by 
 $\sigma$, which in turn extends to a colouring  of the regions determined by $f$ and induced by 
  $g(\sigma)$. On the other hand, since $f\circ g$ is forest we also have a naturally induced colouring of the regions determined by $f\circ g$ and 
  $\sigma$. The claim now follows by the uniqueness of the procedure.
 \end{proof}
 \begin{definition} The category $\vec{\mathfrak F}$ has as objects all pairs of the form $(n,\sigma)$ where $\sigma$ is an $n$-sign and the morphisms from
 $(m,\sigma_1)$ to $(n,\sigma_2)$ are the pairs $(f,\sigma_1)$ where $f(\sigma_1)=\sigma_2$. Composition of morphisms is obvious.
 \end{definition}
 \begin{proposition}
Consider $|=(1,\sigma_0)$, with $\sigma_0$ being the unique $1$-sign. Then $\vec{\mathfrak F}$ satisfies the Unit, Stabilisation and Cancellation properties.  \end{proposition}
 \begin{proof} 
 We start with the Unit property. A morphism in ${\it Mor}_{ \vec{\mathfrak F} }\big( |, (n,\sigma) \big)$ is just a binary tree $f$ such that $f(\sigma_0)=\sigma$. The set of these trees is non-empty since one can always determine a planar binary tree for a  given $n$-sign as done in \cite[Section 5.2, p. 31]{jo1}.

Now we take care of the Stabilisation property. We recall that one can describe the elements of the oriented Thompson group $\vec{F}$ 
 as pairs of trees (see \cite{jo1} for the notation, cf. \cite{CFP}). Given two elements $g(T_1^+,T_1^-)$ and $g(T_2^+,T_2^-)$, if one wants to perform the multiplication $g(T_1^+,T_1^-)g(T_2^+,T_2^-)$, one needs to find other representatives of the factors of the form $g(T_1^{+'},T')$ and $g(T',T_1^{-'})$. This can be done by adding pairs of opposing carets. Then the product is $g(T_1^+,T_1^-)g(T_2^+,T_2^-)=g(T_1^{+'},T_2^{-'})$. In the same way, in the present situation, we may add "carets" and vertical lines forming the forests/morphisms $p$ and $q$ to the trees/morphisms $f$ and $g$ and obtain the same tree $pf=qg$.  
 
 The Cancellation property is clear.
 \end{proof}
 
 \begin{proposition} \label{OF}
 The group of fractions of $\vec{\mathfrak F}$ is naturally isomorphic to \OF.
 \end{proposition}
 \begin{proof} A pair of morphisms in $\vec{\mathfrak F}$ from $|$ to any fixed object is, by definition, a pair of binary planar trees together with a bi-colouring of the
 graph described in \cite{jo1}.  Thus forgetting the $n$-signs defines a homomorphism from the group of fractions of $\vec{\mathfrak F}$ onto \OF which is injective since
 the planar tree  $f$ determines $f(\sigma_0)$.
 \end{proof}
 
 \section{The HOMFLYPT functor.}
 
 Let $\mathfrak C$ be the category whose objects are  sequences $\nu$ of an odd number of signs and 
 whose morphisms $C_\nu^\eta$ 
 are the oriented Conway tangles in  a rectangle with $m$ points on the bottom and $n$ points on
 the top so that the string incident on the $ith$ point on the bottom is ingoing or outgoing if $\nu(i)$ is $-$ or $+$ respectively.  And the string incident on the $ith$ point on the top is ingoing or outgoing if $\eta(i)$ is $+$ or $-$ respectively. Composition is stacking of tangles. Here is an element of $C_\nu^\eta$: \\
 $$
\includegraphics[scale=0.25]{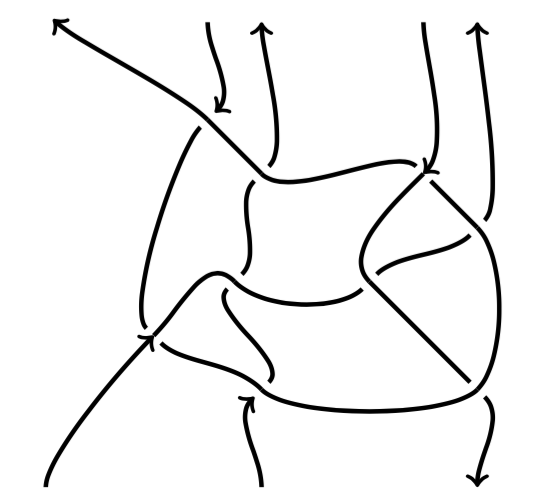}
$$
 where $\nu=(+,+,-)$ and $\eta=(+,-,+,-,+)$.
 
 Put $\delta=\frac{ s^{-2k}-s^{2k}}{s-s^{-1}}$.
 We will only use tangles such that  the leftmost bottom and top strings are ingoing and outgoing respectively which is
 why $n$-signs are required to have $\sigma(1)=+$.

 We now suppose $s=e^{\pi i/r}$, with  $r$ as in Proposition \ref{prop-inner-prod}, and take the positive definite HOMFLYPT skein of $\mathfrak C$ and let $\mathfrak H$ be the 
 category with objects being the $n$-signs as before and morphisms $H_\sigma^\tau$ being the quotient of the vector space
 spanned by $C_\sigma^\tau$ by the kernel of the positive semidefinite form defined by the HOMFLYPT polynomial.
 \begin{definition}
 We define the functor $\Phi : \vec{\mathfrak F} \rightarrow \mathfrak H$ by identifying the objects and assigning an element of ${\it Mor}_{\mathfrak H}(\sigma,\tau)$ to
 a pair $(f,\sigma)$ with $f(\sigma)=\tau$ as follows: \\
 
 If $f$ and $\sigma$ are represented as \\
 
$$
\includegraphics[scale=0.3]{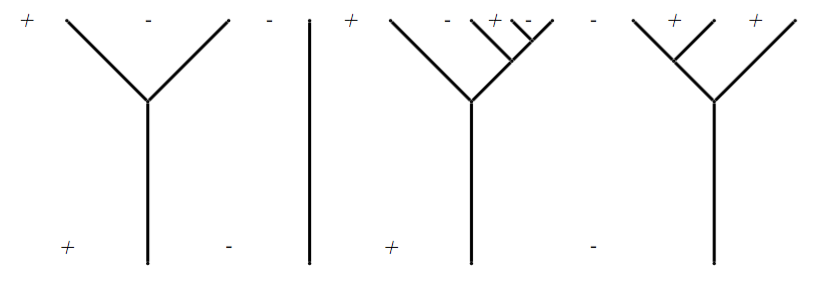}
$$
($\sigma = (+,-,+,-)$, $\tau=(+,-,-,+,-,+,-,-,+,+)$)
 then we form the oriented surface (shaded regions) as follows:\\
$$
\includegraphics[scale=0.45]{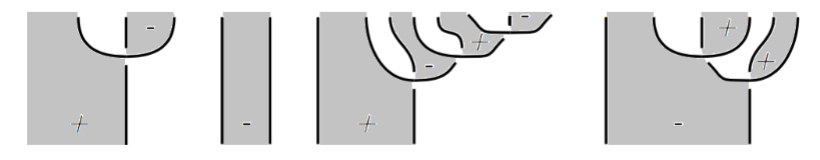}
$$
 And finally we orient the boundary components of the surface to obtain the oriented Conway tangle:\\
$$
\includegraphics[scale=0.4]{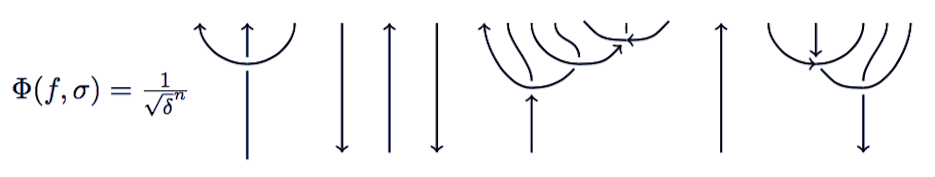}
$$
 which we interpret as a morphism $\Phi(f,\sigma)$ in $\mathfrak H$. 
 We have multiplied by a power of $\delta$ to ensure that $\Phi(f,\sigma)$ is an isometry for the Hilbert space structure.

 \end{definition}
 
 By the procedure of \cite{jo2} (see Proposition 2.1.3 therein) this gives a unitary representation of the group \OF as the group of fractions of $\vec{\mathfrak F}$. 
 
 Let $\mathcal L(g)$ denote the oriented link assigned to $g\in\;$\OF in \cite{jo1}. 
 We recognise the coefficients $\langle \pi(g)\Omega,\Omega\rangle$ of the vacuum vector $\Omega=\; \uparrow$
 in the semicontinuous limit
 as the suitably normalised HOMFLYPT polynomial of $\mathcal L(g)$ for $g\in\;$\OF.
 
 \begin{theorem} \label{HOMFLYPT}
 If $n$ is the number of leaves of $g$, the function $$g\mapsto \frac{1}{\delta^n}{\rm HOMFLYPT}(\mathcal L(g))$$
 is a positive definite function on \OF.
 \end{theorem} 
 
 \begin{remark} If one changes the $*$-structure as in  Remark \ref{real}, the HOMFLYPT polynomial of the \emph{alternating} link
 obtained by the inner product becomes a positive definite function on \OF for $s\in \mathbb R^+$ (actually ${\mathbb R}^*$).
 \end{remark}
 
 \begin{remark} The result also follows for \emph{any} unitary TQFT invariant such as the Kauffman polynomial
 at various values and all cablings of such. And of course for the original Jones polynomial, improving the corresponding partial results obtained in \cite{AC2} by a different method. 
 We gave the details for HOMFLYPT by way of illustration.
 \end{remark}
 
 \smallskip
\noindent{\bf Acknowledgements.} The authors would like to thank the Isaac Newton Institute for Mathematical Sciences, Cambridge, for support and hospitality during the programme \emph{Operator algebras: subfactors and their applications} where work on this paper was undertaken. This work was supported by EPSRC grant no EP/K032208/1.

\end{document}